\documentclass[11pt,a4paper]{amsart}

\usepackage{amssymb,amsfonts,amsmath}

\usepackage[latin1]{inputenc}
\usepackage{color}

\newtheorem{theorem}{Theorem}
\newtheorem{proposition}[theorem]{Proposition}

\newtheorem{corollary}[theorem]{Corollary}

\newtheorem{remark}[theorem]{Remark}
\newtheorem{definition}[theorem]{Definition}

\def\Even{\mathop{\rm Even}\nolimits}
\def\chr{\mathop{\rm chr}\nolimits}
\def\fchr{\mathop{\rm fchr}\nolimits}

\makeatletter
\@namedef{subjclassname@2010}{%
\textup{2010} Mathematics Subject Classification}
\makeatother

\begin{document}

\title[Choosability of paths and free-choosability of cycles]{Choosability of a weighted path and free-choosability of a cycle}
\author{Yves Aubry, Jean-Christophe Godin and Olivier Togni}
\address{Institut de Math\'ematiques de Toulon, Universit\'e du Sud Toulon-Var,  France\\
and Laboratoire LE2I, Universit\'e de Bourgogne, France}

\email{yves.aubry@univ-tln.fr,  godinjeanchri@yahoo.fr
and olivier.togni@u-bourgogne.fr}

\subjclass[2010]{05C15, 05C38, 05C72}

\keywords{Coloring, Choosability, Free-choosability, Cycles.}

\date{\today}

\begin{abstract}
A graph $G$ with a list of colors $L(v)$ and weight $w(v)$ for each vertex $v$ is $(L,w)$-colorable if one can choose a subset of $w(v)$ colors from $L(v)$ for each vertex $v$, such that adjacent vertices receive disjoint color sets. In this paper, we give necessary and sufficient conditions for a weighted path to be $(L,w)$-colorable for some list assignments $L$.
Furthermore, we solve the problem of the free-choosability of a cycle.
\end{abstract}

\maketitle

\section{Introduction}
The concept of choosability of a graph, also called list coloring,  has been introduced by Vizing in \cite{viz76}, and independently by Erd\H{o}s, Rubin and Taylor in \cite{erdo79}. It contains of course the colorability as a particular case. Since its introduction, choosability has been extensively studied (see for example \cite{alontuzavoigt97, borodinkostochkawoodall97, thom03, tuz96, gravier96} and more recently \cite{GutnerTarsi, hav09}).
Even for the original (unweighted) version, the problem proves to be difficult, and is NP-complete for very restricted graph classes. Existing results for the weighted version mainly concern the case of constant weights (i.e. $(a,b)$-choosability), see \cite{alontuzavoigt97, erdo79, GutnerTarsi, tuz96}. 
For the coloring problem of weighted graphs, quite a little bit more is known, see~\cite{mcd00, hav01, hav02, KT06}.

This paper considers list colorings of weighted graphs by studying conditions on the list assignment for a weighted path to be choosable. 
Starting from the idea that in a path, the lists of colors of non consecutive vertices do not interfere, and following the work in \cite{god09}, we introduce here the notion of a waterfall list assignment of a weighted path. It is a list assignment such that any color is present only on one list or on two lists of consecutive vertices. We show that any list assignment (with some additional properties) can be transformed into a similar waterfall list assignment.
Then, using the result of Cropper et al.~\cite{CGHHJ} about Hall's condition for list multicoloring, we prove a necessary and sufficient condition for a weighted path with a given waterfall list $L$ to be $(L,w)$-colorable (Theorem \ref{theolistecascadechoisissable}) and use it to derive $(L,w)$-colorability results for some general lists assignments.

In 1996, Voigt considered the following problem:
let $G$ be a graph and $L$ a list assignment and assume that an arbitrary vertex $v\in V(G)$ is precolored by a color $f\in L(v)$. Is it always possible to complete this precoloring to a proper list coloring ? This question leads to the concept of free-choosability introduced by Voigt in \cite{voi96}. 

We investigate here the free-choosability of the first interesting case, namely the cycle. As an application of Theorem \ref{theolistecascadechoisissable}, we prove our second main result which gives a necessary and sufficient condition for a cycle to be $(a,b)$-free-choosable (Theorem \ref{theorem52these}). In order to get  a concise statement, we introduce the free-choice ratio of a graph, in the same way that Alon, Tuza and Voigt in \cite{alontuzavoigt97} introduced the choice ratio (which equals the so-called fractional chromatic number).

In addition to the results obtained in this paper, the study of waterfall lists may be of more general interest. For now on, the method is extended in~\cite{AGT10} to be used in a reduction process, allowing to prove colorability results on triangle-free induced subgraphs of the triangular lattice.

We recall in Section \ref{Definitions} some definitions related to choosability and free-choosability and introduce the definitions of the similarity between two lists and of a waterfall list that are fundamental for this paper. 
In Section \ref{waterfall}, we show how to transform a list into a similar waterfall list and present a necessary and sufficient condition for a weigthed path to be choosable. Theses result are used in Section \ref{cycle} to obtain conditions for the $(L,w)$-colorability of a weighted path and for the $(a,b)$-free-choosability of a cycle.




\section{Definitions and Preliminaries}\label{Definitions}
Let $G=(V(G),E(G))$ be a graph  where $V(G)$ is the set of vertices and $E(G)$ is the set of edges, and let $a$, $b$, $n$ and $e$ be integers.

Let $w$ be a weight function of $G$ i.e. a map
$w : V(G) \rightarrow {\mathbb N}$ and let $L$ be a list assignment of $G$ i.e. a map
$L : V(G) \rightarrow \mathcal{P}({\mathbb N})$. By abuse of language and to simplify, we will just call $L$ a list.
If $A$ is a finite set, we denote by $\vert A\vert$ the cardinal of $A$.

A weighted graph $(G,w)$ is a graph $G$ together with a weight function $w$ of $G$.

Let us recall the definitions of an $(L,w)$-colorable graph and an $(a,b)$-free-choosable graph which are essential in this paper.

\begin{definition}
An $(L,w)$-coloring $c$ of a graph $G$ is a map that associate to each vertex $v$ exactly $w(v)$ colors from $L(v)$ such that adjacent vertices receive disjoints color sets, i.e. for all  $v \in V(G)$: 
$$c(v) \subset L(v),\ \ \  | c(v) | =w(v),$$ and for all $vv' \in E(G)$:   $$c(v) \cap c(v')  = \emptyset.$$

We say that  $G$ is $(L,w)$-colorable if there exists an $(L,w)$-coloring $c$ of $G$.
\end{definition}

Particular cases of $(L,w)$-colorability are of great interest. In order to introduce them, we define $(L,b)$-colorings and $a$-lists.

An $(L,b)$-coloring $c$ of $G$ is an $(L,w)$-coloring of $G$ such that 
for all $v \in V(G)$, we have $w(v)=b.$

A $a$-list $L$ of $G$ is a list of $G$ such that for all $v \in V(G)$, we have  $|L(v) |=a.$


\begin{definition}
$G$ is said to be $(a,b)$-choosable if for any $a$-list $L$ of $G$, there exists an $(L,b)$-coloring $c$ of $G$.
\end{definition}

\begin{definition}
$G$ is said to be $(a,b)$-free-choosable if for any $v_0 \in V(G)$, and for any  list $L$ of $G$ such that for any $v \in V(G) \setminus \{v_0\},$ we have $|L(v)|=a$ and $|L(v_0)|=b$, there exists an  $(L,b)$-coloring $c$ of $G$.
\end{definition}



We define now the similarity of two lists with respect to a weighted graph:

\begin{definition} Let $(G,w)$ be a weighted graph. Two lists $L$ and $L'$ are said to be {\em similar} if this assertion is true:\\
$$G \text{ is }(L,w)\text{-colorable }\Leftrightarrow G  \text{ is }(L',w)\text{-colorable}.$$
\end{definition}

\medskip

The \textsl{path} $P_{n+1}$ of length $n$ is the graph with vertex set $V=\{v_{0},v_{1},\dots,v_{n}\}$ and edge set $E=\bigcup_{i=0}^{n-1} \{v_{i}v_{i+1}\}$. To simplify the notations, $L(i)$ denotes $L(v_{i})$ and $c(i)$ denotes $c(v_{i})$.\\ 

By analogy with the flow of water in waterfalls, we define a waterfall list as follows:
\begin{definition}
A {\em waterfall} list $L$ of a path $P_{n+1}$ of length $n$ is a list $L$ such that for all $i,j \in \{0,\dots,n\}$ 
with $|i-j| \geq 2$, we have $L(i) \cap L(j) = \emptyset$.
\end{definition}
Notice that another similar definition of a waterfall list is that any color is present only on one 
list or on two lists of consecutive vertices. Figure 1 shows a list $L$ of the path $P_5$ (on the left), together with a similar waterfall list $L^c$ (on the right).

\begin{definition} For a weighted path $(P_{n+1},w)$,
\begin{itemize}
\item A list $L$ is {\em good} if $|L(i)|\geq w(i) + w (i+1)$ for any $i, 1\leq i\leq n-1$.
\item The {\em amplitude} $A(i,j)(L)$ (or $A(i,j)$) of a list $L$ is $A(i,j)(L)=\cup_{k=i}^{j}L(k)$.
\end{itemize}
\end{definition}

\unitlength=0.8cm
\begin{picture}(20,8)
\put(1.8,4.5){\framebox(0.5,2.3){}}
\put(2.6,3.9){\framebox(0.5,2.3){}}
\put(3.4,6.2){\framebox(0.5,0.5){}}
\put(3.4,4.2){\framebox(0.5,1.5){}}
\put(3.4,2.5){\framebox(0.5,0.3){}}
\put(4.2,6.5){\framebox(0.5,0.3){}}
\put(4.2,5.4){\framebox(0.5,0.5){}}
\put(4.2,4.5){\framebox(0.5,0.4){}}
\put(4.2,3){\framebox(0.5,1){}}
\put(4.2,2.3){\framebox(0.5,0.2){}}
\put(5,6.4){\framebox(0.5,0.3){}}
\put(5,5.6){\framebox(0.5,0.3){}}
\put(5,4.8){\framebox(0.5,0.2){}}
\put(5,4.1){\framebox(0.5,0.4){}}
\put(5,3.5){\framebox(0.5,0.2){}}
\put(5,2.9){\framebox(0.5,0.3){}}
\put(5,2.1){\framebox(0.5,0.5){}}
\put(0.1,7){$L( )$}
\put(1.8,7){0}
\put(2.6,7){1}
\put(3.4,7){2}
\put(4.2,7){3}
\put(5,7){4}
\put(7.5,4){\vector(1,0){1}}
\put(8.5,4){\vector(-1,0){1}}
\put(7.3,3.5){similar}
\put(7.5,7){$L^{c}( )$}
\put(9.8,4.5){\framebox(0.5,2.3){}}
\put(10.6,3.9){\framebox(0.5,2.3){}}
\put(11.4,2.3){\framebox(0.5,2.1){}}
\put(12.2,1.8){\framebox(0.5,2.1){}}
\put(13,0.5){\framebox(0.5,1.5){}}
\put(9.8,7){0}
\put(10.6,7){1}
\put(11.4,7){2}
\put(12.2,7){3}
\put(13,7){4}
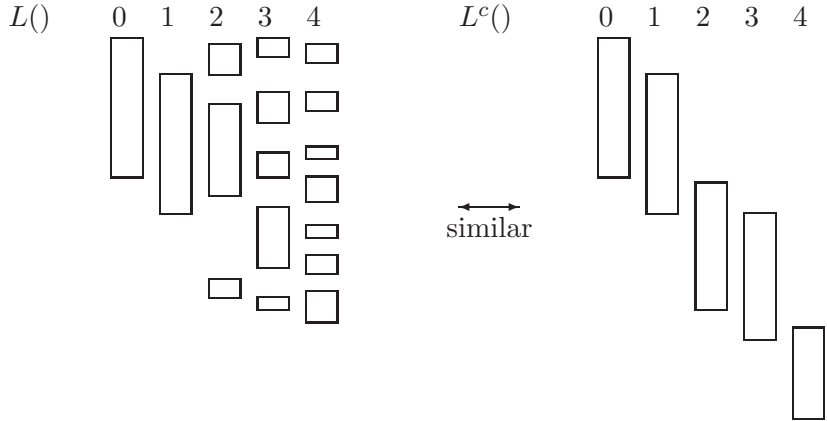
\put(1,0){Fig. 1. Example of a list $L$ which is similar to a waterfall list $L^c$.}
\end{picture}
\bigskip

In \cite{CGHHJ}, Cropper et al. consider Philip Hall's theorem on systems of distinct representatives and its improvement by Halmos and Vaughan as statements about the existence of proper list colorings or list multicolorings of complete graphs. The necessary and sufficient condition in these theorems is generalized in the new setting as "Hall's condition'' : 
$$\forall H\subset G, \sum_{k\in C} \alpha(H,L,k) \ge \sum_{v\in V(H)} w(v),$$ where $C=\bigcup_{v\in V(H)}L(v)$ and $\alpha(H,L,k)$ is the independence number of the subgraph of $H$ induced by the vertices containing $k$ in their color list. Notice that $H$ can restricted to be a connected induced subgraph of $G$.

It is easily seen that Hall's condition is necessary for a graph to be $(L,w)$-colorable. Cropper et al. showed that the condition is also sufficient for some graphs, including paths:

\begin{theorem}[\cite{CGHHJ}]
\label{Cropper} For the following graphs, Hall's condition is sufficient to ensure an $(L,w)$-coloring:
\begin{itemize}
 \item[(a)] cliques;
 \item[(b)] two cliques joined by a cut-vertex;
 \item[(c)] paths;
 \item[(d)] a triangle with a path of length two added at one of its vertices;
 \item[(e)] a triangle with an edge added at two of its three vertices.
\end{itemize}
\end{theorem}

This result is very nice, however, it is often hard to compute the left part of Hall's condition, even for paths. Hence, for our study on choosability of weigthed paths, we find convenient to work with waterfall lists for which, as we will see in the next section, Hall's condition is very easy to check.

\section{waterfall lists}\label{waterfall}

We first show that any good list can be transformed into a similar waterfall list.
\begin{proposition}\label{similar}
\label{llf}
 For any good list $L$ of $P_{n+1}$, there exists a similar waterfall list $L^c$ with $\vert L^c(i)\vert=\vert L(i)\vert$ for all $i\in\{0,\ldots,n\}$.
\end{proposition}

\begin{proof}
We are going to transform a good list $L$ of $P_{n+1}$  into a waterfall list $L^c$ and we will prove that $L^c$ is similar with $L$.

First, remark that if a color $x\in L(i-1)$ but $x\not\in L(i)$ for some $i$ with  $1\leq i \leq n-1$, then for any $j>i$, one can change the color $x$ by a new color 
$y\not\in A(0,n)(L)$ in the list $L(j)$, without changing the choosability of the list.
With this remark in hand, we can assume that $L$ is such that any color $x$ appears on the lists of consecutive vertices $i_x,\ldots , j_x$.

Now, by permuting the colors if necessary, we can assume that if $x<y$ then $i_x < i_y$ or $i_x =i_y$ and $j_x\leq j_y$.

Repeat the following transformation:

1. Take the minimum color $x$ for which $j_x \geq i_x +2$ i.e. the color $x$ is present on at least three vertices 
$i_x, i_x +1 ,i_x +2,\ldots , j_x$;

2. Replace color $x$ by a new color $y$ in lists $L(i_x +2),\ldots , L(j_x )$;

\noindent
until the obtained list is a waterfall list (obviously, the number of iterations is always finite).

Now, we show that this transformation preserves the choosability of the list.
Let $L'$ be the list obtained from the list $L$ by the above transformation.

If $c$ is an $(L,w )$-coloring of $P_{n+1}$  then the coloring $c'$ obtained from $c$ by changing the color $x$ by the color $y$ in the color set 
$c(k)$ of each vertex $k\geq i_x +2$ (containing $x$) is an  $(L',w )$-coloring since $y$ is a new color.

Conversely, if $c'$ is an $(L',w )$-coloring of $P_{n+1}$, we consider two cases:

\medskip
\noindent
{\sl Case 1}: $x\not\in c'(i_x +1)$ or $y\not\in c'(i_x +2)$. In this case, the coloring $c$ obtained from $c'$ by changing the color $y$ by the 
color $x$ in the color set $c'(k)$ of each vertex $k\geq i_x +2$ (containing $y$) is an  $(L,w )$-coloring.

\medskip
\noindent
{\sl Case 2}: $x\in c'(i_x +1)$ and $y\in c'(i_x +2)$. We have to consider two subcases:
\begin{itemize}
 \item Subcase 1: $L'(i_x +1 )\not\subset (c'(i_x ) \cup c'(i_x +1) \cup c'(i_x +2))$. 
There exists $z \in L'(i_x +1 )\setminus (c'(i_x ) \cup c'(i_x +1) \cup c'(i_x +2))$ and the coloring $c$ obtained from $c'$ by changing the color 
$x$ by the color $z$ in $c'(i_x +1)$ and replacing the color $y$ by the color $x$ in the color set $c'(k)$ of each vertex $k\geq i_x +2$ (containing $y$) is an  $(L,w )$-coloring.
 \item Subcase 2: $L'(i_x +1 )\subset (c'(i_x ) \cup c'(i_x +1) \cup c'(i_x +2))$. We have
$$|L'(i_x +1)| = \Big| \Big( (c'(i_x ) \cup c'(i_x +1) \cup c'(i_x +2)\Big) \cap L'(i_x +1)\Big|.  $$ 
As $c'$ is an $(L',w)$-coloring of $P_{n+1}$, we have
$$ |L'(i_x +1)|=|c'(i_x +2)  \cap L'(i_x +1)| + |c'(i_x +1)  \cap L'(i_x +1)|+\Big|\Big( c'(i_x) \backslash c'(i_x +2) \Big) \cap L'(i_x +1)\Big|,$$
$$ |L'(i_x +1)| - w (i_x +1) - |c'(i_x+2)  \cap L'(i_x +1)| = \Big|\Big( c'(i_x) \backslash c'(i_x +2) \Big) \cap L'(i_x +1)\Big|.  $$

Since $y \in c'(i_x +2)$ and $y \notin L'(i_x +1)$, we obtain that 
$$|c'(i_x +2)  \cap L'(i_x +1)| \leq w (i_x +2) - 1,$$ hence
$$ \Big( |L'(i_x +1)| - w (i_x +1) - w (i_x +2) \Big) + 1 \leq \Big| \Big( c'(i_x) \backslash c'(i_x +2) \Big) \cap L'(i_x +1)\Big|.$$
But, by hypothesis, $L$ is a good list. Thus $|L(i_x +1)|=|L'(i_x +1)| \geq w(i_x +1) + w(i_x +2)$ and 
$$ 1 \leq \Big|\Big( c'(i_x) \backslash c'(i_x +2) \Big) \cap L'(i_x +1)\Big| .$$
Consequently, there exists $z \in \Big( c'(i_x) \backslash c'(i_x +2) \Big) \cap L'(i_x +1)$. The coloring $c$ is then constructed from $c'$ by changing
the color $x$ by the color $z$ in $c'(i_x +1)$, the color $z$ by the color $x$ in $c'(i_x)$ and the color $y$ by the color $x$ in the set $c'(k)$ of each vertex $k\geq i_x +2$.
\end{itemize}

\end{proof}

%


The following theorem, which is a corollary of Theorem~\ref{Cropper}, gives a necessary and sufficient condition for a weighted path to be 
$(L^{c},w)$-colorable where $L^c$ is a waterfall list.

\begin{theorem}
\label{theolistecascadechoisissable}
Let $L^{c}$ be a waterfall list of a weighted path $(P_{n+1},w)$. 
Then $P_{n+1}$ is $(L^{c},w)$-colorable if and only if:
$$\forall i,j \in \{0,\dots,n\},  \: |\bigcup_{k=i}^{j}L^{c}(k)| \geq \sum_{k=i}^{j} w (k).$$ 
\end{theorem}

\begin{proof}
``if'' part: Recall that $A(i,j)=\cup_{k=i}^{j}L^c(k)$.
For $i,j \in \{0,\dots,n\}$, let $P_{i,j}$ be the subpath of $P_{n+1}$ induced by the vertices $i,\ldots, j$.
By Theorem \ref{Cropper}, it is sufficient to show that 
$$\forall i,j \in \{0,\dots,n\},  \: \sum_{x\in A(i,j)} \alpha(P_{i,j},L^c,x) \geq \sum_{k=i}^{j} w (k).$$
Since the list is a waterfall list, then for each color $x\in A(i,j)$, $\alpha(P_{i,j},L^c,x)=1$ and thus 
$\sum_{x\in A(i,j)} \alpha(P_{i,j},L^c,x) = |A(i,j)| = |\bigcup_{k=i}^{j}L^{c}(k)|$.

``only if'' part: If $c$ is a $(L^{c},w)$-coloring of $P_{n+1}$  then
$$\forall i,j \in \{0,\dots,n\} : \: \bigcup_{k=i}^{j}L^{c}(k) \supset \bigcup_{k=i}^{j} c(k).$$
Since $L^c$ is a waterfall list, it is easily seen that $|\bigcup_{k=i}^{j} c(k)|=\sum_{k=i}^{j} w(k)$. Therefore, 
$\forall i,j \in \{0,\dots,n\} : \: |\bigcup_{k=i}^{j}L^{c}(k)| \geq \sum_{k=i}^{j} w (k)$.
\end{proof}

\section{Choosability of a path and free-choosability of a cycle}\label{cycle}

Theorem~\ref{theolistecascadechoisissable} has the following corollary when the list is a good waterfall list and 
 $|L(n)| \geq w(n)$. 

\begin{corollary}
\label{lemmeencascadeequige}
Let $L^{c}$ be a waterfall list of a weighted path $(P_{n+1},w)$ such that for any 
$i, 1\leq i\leq n-1$, $|L^{c}(i)|\geq w(i) + w (i+1)$ and
$|L^c(n)| \geq w(n)$. Then $P_{n+1}$ is  $(L^{c},w)$-colorable if and only if  
$$ \forall j \in \{0,\dots,n\}, \:|\bigcup_{k=0}^{j}L^{c}(k)| \geq \sum_{k=0}^{j} w (k). $$ 
\end{corollary}

\begin{proof}
Under the hypothesis, if $P_{n+1}$ is $(L^{c},w)$-colorable, then Theorem 
 \ref{theolistecascadechoisissable} proves in particular the result.\\
 
 Conversely, since
 $L^{c}$ is a waterfall list of  $P_{n+1}$, we have:
$$\forall i,j \in \{1,\dots,n\}, \: |A(i,j)|=|\cup_{k=i}^{j} L^c(k)|  \geq |\cup_{\substack{k=i \\ k-i \ even}}^{j} L^c(k)| =\sum_{\substack{k=i \\ k-i  \ even}}^{j} |L^{c}(k)|. $$
Since $L^c$ is a good list of   $P_{n+1}$  (for simplicity, we set   $w(n+1)=0$):
$$ \forall i,j \in \{1,\dots,n\}, \: \sum_{\substack{k=i \\ k-i  \ even}}^{j} |L^{c}(k)| \geq \sum_{\substack{k=i \\ k-i  \ even}}^{j} (w (k) + w (k+1))  \geq \sum_{k=i}^{j} w (k),$$
then we obtain for all  $ i,j \in \{1,\dots,n\}, \: |A(i,j)| \geq \sum_{k=i}^{j} w (k)$. Since for all $ j \in \{0,\dots,n\}, \:|A(0,j)| \geq \sum_{k=0}^{j} w (k) $, Theorem \ref{theolistecascadechoisissable} concludes the proof.
\end{proof}

\bigskip

Another interesting corollary holds for lists $L$ such that 
 $|L(0)|=|L(n)|= b$, and for all $i \in \{1,\dots,n-1\},  |L(i)|=a$. The function $\Even$ is defined for any real $x$ by: $\Even(x)$ is the smallest even integer $p$ such that $p\geq x$.

\begin{corollary}
\label{theorem48these}
Let $L$ be a list of $P_{n+1}$ such that $|L(0)|=|L(n)|= b$,
and $|L(i)|=a=2b+e$ for all $i \in \{1,\dots,n-1\}$ (with $e\not=0$). 
\begin{center}
If $n \geq \Even\Bigl(\frac{2b}{e}\Bigr)$ then $P_{n+1}$ is $(L,b)$-colorable.                                                                              \end{center}
\end{corollary}

\begin{proof}
The hypothesis implies that $L$ is a good list of $P_{n+1}$, hence by Proposition \ref{llf}, there exists a waterfall list $L^c$ similar to $L$. So we get:
$$\forall i \in \{1,\dots,n-1\}, \ |L^{c}(i)| \geq 2b = w (i) + w (i+1)$$
and $|L^{c}(n)| \geq b=w (n)$. By Corollary  \ref{lemmeencascadeequige} it remains to prove that: 
$$\forall j \in \{0,\dots,n\},  \:|A(0,j)| \geq \sum_{k=0}^{j} w (k) = (j+1)b.$$ 

{\sl Case 1}: $j=0$. By hypothesis, we have  $|A(0,0)|=|L^{c}(0)| \geq b$.\\

{\sl Case 2}: $j \in \{1,\dots,n-1\}$. Since $L^{c}$ is a waterfall list of $P_{n+1}$ we obtain that: \\
if $j$ is even
$$|A(0,j)| \geq \sum_{\substack {k=0 \\ k\ even}}^j |L^c(k)| = b + \sum_{\substack{k=2 \\ k  \ even}}^{j} 2b = b + \frac{j}{2} 2b = (j+1)b,$$
and if $j$ is odd
$$ |A(0,j)| \geq  \sum_{\substack {k=0 \\ k\ odd}}^j |L^c(k)| = \sum_{\substack{k=1 \\ k  \ odd}}^{j} 2b = \frac{j+1}{2} 2b =(j+1)b.  $$
Hence for all $j \in \{0,\dots,n-1\}, \ |A(0,j)| \geq  (j+1)b$.

{\sl Case 3}: $j=n$. Since $n \geq \Even\Bigl(\frac{2b}{e}\Bigr)$ by hypothesis, and 
$$\mid A(0,n) \mid \geq \sum_{\substack {k=0 \\ k\ odd}}^n |L^c(k)|=\left \{
\begin{array}{l}
a\frac{n}{2}  \ \  \ \ \ \  \ \ \ \ if  \ n \ is \ even \\
b+a\frac{n-1}{2}   \  \ \ \  \ \ \ \  \ otherwise
\end{array}
\right.
$$
we deduce that $|A(0,n)| \geq (n+1)b$, which concludes the proof.
\end{proof}

\bigskip

For example, let $P_{n+1}$ be the path of length $n$ with a list $L$ such that $|L(0)|=|L(n)|= 4$,
and $|L(i)|=9$ for all $i \in \{1,\dots,n-1\}$. Then the previous Corollary tells us that we can find an $(L,4)$-coloring of $P_{n+1}$ whenever $n\geq 8$. In other words, if $n\geq 8$, we can choose 4 colors on each vertex such that adjacent vertices receive disjoint colors. If  $|L(i)|=11$ for all $i \in \{1,\dots,n-1\}$, then $P_{n+1}$ is $(L,4)$-colorable whenever $n\geq 4$.

\medskip
The above result is a starting tool used in \cite{AGT10} to attack McDiarmid and Reed's conjecture claiming that every triangle free induced subgraph of the triangular lattice is $(9,4)$-colorable (hence the values $a=9$ and $b=4$ are somehow ``natural'').
It is also used in the following to determine the free-choice-ratio of the cycle.




\bigskip

The cycle $C_{n}$  of length $n$ is the graph with vertex set  $V=\{v_{0},\dots,v_{n-1}\}$ and edge set $E=\bigcup_{i=0}^{n-1} \{v_{i}v_{i+1 (mod \ n)}\}$.

\bigskip

Let $F\mathcal{C}_h(x)$ be the set of graphs $G$ which are $(a,b)$-free-choosable for all $a,b$ such that $\frac{a}{b} \geq x$:

$$F\mathcal{C}_h(x)=\{ G\ | \ \forall \ \frac{a}{b} \geq x, \ G \ \text{is} \ (a,b)\text{-free-choosable}\}.$$

Moreover, we can define the free-choice ratio 
 $\fchr(G) $ of a graph $G$ by:
 
$$\fchr(G):=\inf \{\frac{a}{b} \ | \ G \  \text{is}\ (a,b)\text{-free-choosable} \}.$$

If $\lfloor x \rfloor$ denotes the greatest integer less or equal to the real $x$,  we can state:


\begin{theorem}
\label{theorem52these}
If $C_n$ is a cycle of length $n$, then 
$$C_n \in F\mathcal{C}_h(2+ \Big\lfloor \frac{n}{2} \Big\rfloor ^{-1}). $$
Moreover, we have:
$$\fchr(C_n)=2+ \Big\lfloor \frac{n}{2} \Big\rfloor ^{-1}.$$
\end{theorem}

\begin{proof}
Let $a,b$ be two integers such that  $a/b \geq  2+\lfloor \frac{n}{2}\rfloor ^{-1}$. Let $C_n$ be a cycle of length $n$ and $L$ a $a$-list of $C_n$. Without loss of generality, we can suppose that $v_0$ is the vertex chosen for the free-choosability and $L_0\subset L(v_0)$ has $b$ elements. 
It remains to construct an $(L,b)$-coloring $c$ of $C_n$ such that $c(v_0)=L_0$.
Hence we have to construct an $(L',b)$-coloring  $c$ of $P_{n+1}$  such that $L'(0)=L'(n)=L_0$ and for all $i \in \{1,...,n-1\}$,  $L'(i)=L(v_i)$. We have 
$|L'(0)|=|L'(n)|=b$ and for all  $i \in \{1,...,n-1\}$,  $|L'(i)|=a$.
Since 
$a/b \geq 2+\lfloor \frac{n}{2}\rfloor ^{-1}$ and 
$e=a-2b$, we get 
$e/b \geq\lfloor \frac{n}{2}\rfloor^{-1}$ hence
$n \geq \Even(2b/e)$.
Using Corollary \ref{theorem48these}, we get:

$$C_n \in F\mathcal{C}_h(2+ \Big\lfloor \frac{n}{2} \Big\rfloor ^{-1}).$$

Hence, we have that $\fchr(C_n) \leq 2+\lfloor \frac{n}{2}\rfloor ^{-1}.$ Moreover, let us prove that 
 $M=2+\lfloor \frac{n}{2}\rfloor ^{-1}$ is reached. \\

For $n$ odd, Voigt has proved in \cite{voi98} that the choice ratio $\chr(C_n)$ of a cycle of odd length $n$ is exactly $M$.
Hence $\fchr(C_n)\geq \chr(C_n)=M$, and the result is proved.

For $n$ even, let   $a,b$ be two integers such that  $\frac{a}{b} < M$. We construct a counterexample for the free-choosability: let $L$ be the list of $C_n$ such that

$$L(i) = \left \{
\begin{array}{l}
\{1,\dots,a\} \  \ \ \  \ \ \ \ \ \ \  \  \ \ \ \ \ \ \ \  \ \ \ \ \ \ \ \  \ \ \ \ \ \ \ \  if \ i \in \{0,1\} \\
\{1+\frac{i-1}{2}a,\dots,(\frac{i-1}{2}+1)a\} \  \ \  \ \ \ \ \ \ \ \ \  if \ i \not=n-1 \ is \ odd \\
\{b+1+\frac{i-2}{2}a,\dots,b+(\frac{i-2}{2}+1)a\}  \  \ \ \ \ \ \ \ if \ i \ is \ even \ and \ i \not= 0 \\
\{1,\dots,b,1+(\frac{n-4}{2}+1)a,\dots,(\frac{n-4}{2}+2)a-b\} \ \ \  \ \ if \ i=n-1
\end{array}
\right.
$$ 

If we choose  $c_0=\{1,\dots,b\} \subset L(0)$, we can check that it does not exist an  $(L,b)$-coloring of $C_n$ such that $c(0)=c_0$, so we could not do better. 
\end{proof}

\begin{remark} In particular, the previous theorem implies that
if $n \geq \Even(\frac{2b}{e})$ then the cycle $C_{n}$ of length $n$ is $(2b+e,b)$-free-choosable.
\end{remark}

\begin{remark} 
Erd\H{o}s, Rubin and Taylor have stated in \cite{erdo79} the following question:
If $G$ is $(a,b)$-colorable, and $\frac{c}{d}>\frac{a}{b}$, does it imply that $G$ is $(c,d)$-colorable ?
Gutner and Tarsi have shown in \cite{GutnerTarsi} that the answer is negative in general.
If we consider the analogue question for free-choosability, then the previous theorem implies that it is true for the cycle.
\end{remark}


\end{document}